\documentclass[a4paper]{amsart}

\RequirePackage{amsmath} \RequirePackage{amssymb}
\usepackage{amscd,latexsym,amsthm,amsfonts,amssymb,amsmath,amsxtra}
\usepackage[colorlinks=true, urlcolor=blue,bookmarks=true,bookmarksopen=true,
citecolor=blue]{hyperref}
\usepackage{fullpage}

\usepackage[all]{xy}
\usepackage[OT2,T1]{fontenc}

\DeclareSymbolFont{cyrletters}{OT2}{wncyr}{m}{n}

\DeclareMathOperator{\chaR}{char}

\DeclareMathOperator{\rank}{rank}
\DeclareMathOperator{\Tr}{Tr}

    \newcommand{\Aut}{{\mathrm{Aut}}}

    \newcommand{\Gal}{{\mathrm{Gal}}} \newcommand{\GL}{{\mathrm{GL}}}
    
    \DeclareMathOperator{\Hom}{Hom}

    \newcommand{\loc}{{\mathrm{loc}}}

    \DeclareMathOperator{\ord}{ord}
     
    \newcommand{\pr}{{\mathrm{pr}}}

    \newcommand{\Sel}{{\mathrm{Sel}}}

    \newcommand{\tr}{{\mathrm{tr}}}

    \newcommand{\wh}{\widehat}

    \newcommand{\ov}{\overline}

    \newcommand{\ra}{\rightarrow}

\newcommand{\xra}{\xrightarrow}

\begin{document}

    \theoremstyle{plain}
    \newtheorem{thm}{Theorem}[section]
    \newtheorem{cor}[thm]{Corollary}
    \newtheorem{lem}[thm]{Lemma}
    \newtheorem{hyp}[thm]{Hypothesis}
    \newtheorem{prop}[thm]{Proposition}
    \newtheorem{conj}[thm]{Conjecture}
    \newtheorem{fact}[thm]{Fact}
    \newtheorem{claim}[thm]{Claim}
    \theoremstyle{definition}
    \newtheorem{defn}[thm]{Definition}
    \newtheorem{example}[thm]{Example}
    \newtheorem{exercise}[thm]{Exercise}
    \theoremstyle{remark}
    \newtheorem{remark}[thm]{Remark}
    \newtheorem{Q}[thm]{Question}
    \renewcommand{\theequation}{\arabic{equation}}
    \numberwithin{equation}{section}

 \newcommand{\Neron}{N\'{e}ron~}\newcommand{\adele}{ad\'{e}le~}
    \newcommand{\Adele}{Ad\'{e}le~}\newcommand{\adeles}{ad\'{e}les~}
    \newcommand{\adelic}{ad\'{e}lic}
    \newcommand{\idele}{id\'{e}le~}\newcommand{\Idele}{Id\'{e}le~}
    \newcommand{\ideles}{id\'{e}les~}\newcommand{\etale}{\'{e}tale~}
    \newcommand{\Etale}{\'{E}tale~}
    \newcommand{\Poincare}{Poincar\'{e}~}\newcommand{\et}{{\text{\rm\'et}}}

    \newcommand{\BA}{{\mathbb {A}}} \newcommand{\BB}{{\mathbb {B}}}
    \newcommand{\BC}{{\mathbb {C}}} \newcommand{\BD}{{\mathbb {D}}}
    \newcommand{\BE}{{\mathbb {E}}} \newcommand{\BF}{{\mathbb {F}}}
    \newcommand{\BG}{{\mathbb {G}}} \newcommand{\BH}{{\mathbb {H}}}
    \newcommand{\BI}{{\mathbb {I}}} \newcommand{\BJ}{{\mathbb {J}}}
    \newcommand{\BK}{{\mathbb {K}}} \newcommand{\BL}{{\mathbb {L}}}
    \newcommand{\BM}{{\mathbb {M}}} \newcommand{\BN}{{\mathbb {N}}}
    \newcommand{\BO}{{\mathbb {O}}} \newcommand{\BP}{{\mathbb {P}}}
    \newcommand{\BQ}{{\mathbb {Q}}} \newcommand{\BR}{{\mathbb {R}}}
    \newcommand{\BS}{{\mathbb {S}}} \newcommand{\BT}{{\mathbb {T}}}
    \newcommand{\BU}{{\mathbb {U}}} \newcommand{\BV}{{\mathbb {V}}}
    \newcommand{\BW}{{\mathbb {W}}} \newcommand{\BX}{{\mathbb {X}}}
    \newcommand{\BY}{{\mathbb {Y}}} \newcommand{\BZ}{{\mathbb {Z}}}

    \newcommand{\CA}{{\mathcal {A}}} \newcommand{\CB}{{\mathcal {B}}}
    \newcommand{\CC}{{\mathcal {C}}} \renewcommand{\CD}{{\mathcal {D}}}
    \newcommand{\CE}{{\mathcal {E}}} \newcommand{\CF}{{\mathcal {F}}}
    \newcommand{\CG}{{\mathcal {G}}} \newcommand{\CH}{{\mathcal {H}}}
    \newcommand{\CI}{{\mathcal {I}}} \newcommand{\CJ}{{\mathcal {J}}}
    \newcommand{\CK}{{\mathcal {K}}} \newcommand{\CL}{{\mathcal {L}}}
    \newcommand{\CM}{{\mathcal {M}}} \newcommand{\CN}{{\mathcal {N}}}
    \newcommand{\CO}{{\mathcal {O}}} \newcommand{\CP}{{\mathcal {P}}}
    \newcommand{\CQ}{{\mathcal {Q}}} \newcommand{\CR}{{\mathcal {R}}}
    \newcommand{\CS}{{\mathcal {S}}} \newcommand{\CT}{{\mathcal {T}}}
    \newcommand{\CU}{{\mathcal {U}}} \newcommand{\CV}{{\mathcal {V}}}
    \newcommand{\CW}{{\mathcal {W}}} \newcommand{\CX}{{\mathcal {X}}}
    \newcommand{\CY}{{\mathcal {Y}}} \newcommand{\CZ}{{\mathcal {Z}}}

    \newcommand{\RA}{{\mathrm {A}}} \newcommand{\RB}{{\mathrm {B}}}
    \newcommand{\RC}{{\mathrm {C}}} \newcommand{\RD}{{\mathrm {D}}}
    \newcommand{\RE}{{\mathrm {E}}} \newcommand{\RF}{{\mathrm {F}}}
    \newcommand{\RG}{{\mathrm {G}}} \newcommand{\RH}{{\mathrm {H}}}
    \newcommand{\RI}{{\mathrm {I}}} \newcommand{\RJ}{{\mathrm {J}}}
    \newcommand{\RK}{{\mathrm {K}}} \newcommand{\RL}{{\mathrm {L}}}
    \newcommand{\RM}{{\mathrm {M}}} \newcommand{\RN}{{\mathrm {N}}}
    \newcommand{\RO}{{\mathrm {O}}} \newcommand{\RP}{{\mathrm {P}}}
    \newcommand{\RQ}{{\mathrm {Q}}} \newcommand{\RR}{{\mathrm {R}}}
    \newcommand{\RS}{{\mathrm {S}}} \newcommand{\RT}{{\mathrm {T}}}
    \newcommand{\RU}{{\mathrm {U}}} \newcommand{\RV}{{\mathrm {V}}}
    \newcommand{\RW}{{\mathrm {W}}} \newcommand{\RX}{{\mathrm {X}}}
    \newcommand{\RY}{{\mathrm {Y}}} \newcommand{\RZ}{{\mathrm {Z}}}

    \newcommand{\fa}{{\mathfrak{a}}} \newcommand{\fb}{{\mathfrak{b}}}
    \newcommand{\fc}{{\mathfrak{c}}} \newcommand{\fd}{{\mathfrak{d}}}
    \newcommand{\fe}{{\mathfrak{e}}} \newcommand{\ff}{{\mathfrak{f}}}
    \newcommand{\fg}{{\mathfrak{g}}} \newcommand{\fh}{{\mathfrak{h}}}
    \newcommand{\fii}{{\mathfrak{i}}} \newcommand{\fj}{{\mathfrak{j}}}
    \newcommand{\fk}{{\mathfrak{k}}} \newcommand{\fl}{{\mathfrak{l}}}
    \newcommand{\fm}{{\mathfrak{m}}} \newcommand{\fn}{{\mathfrak{n}}}
    \newcommand{\fo}{{\mathfrak{o}}} \newcommand{\fp}{{\mathfrak{p}}}
    \newcommand{\fq}{{\mathfrak{q}}} \newcommand{\fr}{{\mathfrak{r}}}
    \newcommand{\fs}{{\mathfrak{s}}} \newcommand{\ft}{{\mathfrak{t}}}
    \newcommand{\fu}{{\mathfrak{u}}} \newcommand{\fv}{{\mathfrak{v}}}
    \newcommand{\fw}{{\mathfrak{w}}} \newcommand{\fx}{{\mathfrak{x}}}
    \newcommand{\fy}{{\mathfrak{y}}} \newcommand{\fz}{{\mathfrak{z}}}
     \newcommand{\fA}{{\mathfrak{A}}} \newcommand{\fB}{{\mathfrak{B}}}
    \newcommand{\fC}{{\mathfrak{C}}} \newcommand{\fD}{{\mathfrak{D}}}
    \newcommand{\fE}{{\mathfrak{E}}} \newcommand{\fF}{{\mathfrak{F}}}
    \newcommand{\fG}{{\mathfrak{G}}} \newcommand{\fH}{{\mathfrak{H}}}
    \newcommand{\fI}{{\mathfrak{I}}} \newcommand{\fJ}{{\mathfrak{J}}}
    \newcommand{\fK}{{\mathfrak{K}}} \newcommand{\fL}{{\mathfrak{L}}}
    \newcommand{\fM}{{\mathfrak{M}}} \newcommand{\fN}{{\mathfrak{N}}}
    \newcommand{\fO}{{\mathfrak{O}}} \newcommand{\fP}{{\mathfrak{P}}}
    \newcommand{\fQ}{{\mathfrak{Q}}} \newcommand{\fR}{{\mathfrak{R}}}
    \newcommand{\fS}{{\mathfrak{S}}} \newcommand{\fT}{{\mathfrak{T}}}
    \newcommand{\fU}{{\mathfrak{U}}} \newcommand{\fV}{{\mathfrak{V}}}
    \newcommand{\fW}{{\mathfrak{W}}} \newcommand{\fX}{{\mathfrak{X}}}
    \newcommand{\fY}{{\mathfrak{Y}}} \newcommand{\fZ}{{\mathfrak{Z}}}

\title{Rubin's conjecture on local units in the anticyclotomic tower at inert primes: $p=3$ case}

\author{Xiaojun Yan}
\address{Beijing Institute of Mathematical Sciences and Applications, Beijing 101408, China.}\address{Department of Mathematics and Yau Mathematical Sciences Center, Tsinghua University;}
\email{xjyan95@amss.ac.cn}

\author{Xiuwu Zhu}
\address{Beijing Institute of Mathematical Sciences and Applications, Beijing 101408, China.}\address{Department of Mathematics and Yau Mathematical Sciences Center, Tsinghua University;}
\email{xwzhu@bimsa.cn}

\date{\today}
\maketitle

\begin{abstract}
We prove Rubin's conjecture on the structure of local units in the anticyclotomic $\BZ_p$-extension of unramified quadratic extension of $\BQ_p$ in $p=3$ case by extending Burungale-Kobayashi-Ota's work.
\end{abstract}

\tableofcontents
\section{Introduction}

\subsection{Background}
Iwasawa theory is a basic tool to study the growth of the Mordell-Weil rank of elliptic curves in a tower of number fields and its relation to special $L$-value.
For an elliptic curve $E$ over $\BQ$ with complex multiplication by an imaginary quadratic field $K$, it is classical to study the module of local units modulo elliptic units attached to $E$ in the $\BZ_p^2$-extension of $K$.
If $p$ splits in $K$, then this module is torsion, and its characteristic ideal is generated by the two-variable Katz $p$-adic $L$-function attached to $E$ (cf. \cite{Ya}).
However, if $p$ inerts in $K$, this module is non-torsion, since the rank of the module of local units is twice that of the module of elliptic units.

Let $\Lambda$ be the Iwasawa algebra for the anticyclotomic $\BZ_p$-extension of an unramified quadratic extension of $\BQ_p$.
Rubin considered the $\Lambda$-module $V$, the anticyclotomic projection of local units of $\BZ_p^2$-extension of $K$, and defined two rank $1$ free submodules $V^{\pm}$.
He conjectured (cf. \cite{Ru}) that
$$V=V^+\oplus V^-.$$
The projection of every elliptic unit lies in $V^\epsilon$, where $\epsilon$ is the sign of $L(E/\BQ,s)$.
Under the conjecture, Rubin constructed a $p$-adic $L$-function, which generates the quotient of $V^\epsilon$ by the image of elliptic units.
Moreover, Agboola-Howard \cite{AH} formulated and proved an Iwasawa main conjecture that involves Rubin's $p$-adic $L$-function under Rubin's conjecture.

Rubin proposed a criterion under which the conjecture is true in the case $p\geq 5$.
His criterion involves the existence of following global objects:
\begin{enumerate}
    \item[(R1)] a CM elliptic curve with good supersingular reduction at $p$ whose central $L$-value is $p$-indivisible,
    \item[(R2)] a Heegner point over imaginary quadratic fields with $p$ inert which is locally $p$-indivisible.
\end{enumerate}
He proved that there are primes $p$ with density $1$ at which (R1) exists.
In \cite{BKO}, using the results of \cite{Fi}, Burungale-Kobayashi-Ota verified the existence of a modified (R1) for primes $p>3$.
For (R2), Rubin verified that it exists for $5\leq p\leq 1000$ and $p\not\equiv 1\pmod{12}$ by using the computation of Stephen (unpublished, but similar to \cite{BS}).
However, in general, it is difficult to verify the local $p$-indivisibility of Heegner points.
Burungale-Kobayashi-Ota consider formal CM points and the modular parametrization of elliptic curves instead of Heegner points.
They constructed such formal CM points when $p>3$, and proved Rubin's conjecture in the case $p>3$.

In this paper, we prove Rubin's conjecture for the case $p=3$ by constructing special formal CM points in this case following Burungale-Kobayashi-Ota's approach.
As an application, we complete the proof of Agboola-Howard's main conjecture when $p=3$.
The result has various potential applications such as extending the $p$-adic Waldspurger formula presented in \cite{burungale2024p} to the prime $p=3$,  \cite{BURUNGALE2024} on Kato's epsilon-conjecture and \cite{burungale2024hecke} on vanishing of $\mu$-invariants on Rubin's $p$-adic $L$-function.

In the case $p=3$, we remark that Rubin's criterion also works and it may be verified by some computational methods.
\subsection{Statement} Let $p$ be a prime.
Let $\Phi$ be the unramified quadratic extension of $\BQ_p$ and $\CO$ be its ring of integers.
Let $\CF_{/\CO}$ be a Lubin-Tate formal group with parameter $\pi:=-p$.
Let $\Phi_n=\Phi(\CF[\pi^{n+1}])$ for $0\leq n\leq \infty$.
Then we have an isomorphism $\kappa: \Gal(\Phi_\infty/\Phi)\xrightarrow{\sim}\CO^\times$, $\sigma\mapsto \kappa(\sigma)$ where $\sigma(v)=[\kappa(\sigma)^{-1}](v)$ for all $v\in \CF[\pi^{\infty}]$.
Let $\Delta$ be the torsion subgroup of $\Gal(\Phi_\infty/\Phi)$.
Let $\Theta_n=\Phi_n^{\Delta}$ for all $n\leq \infty$.

\subsubsection{Coleman power series and Coates-Wiles homomorphism}
For a finite extension $F$ of $\BQ_p$, we denote $U(F)$ its group of local principal units.
Define
$$U_\infty=\left(\varprojlim(U(\Phi_n)\otimes_{\BZ_p} \CO)\right)^{\kappa|_{\Delta}},\quad U_\infty^*=U_\infty\otimes_{\CO}T_{\pi}\CF^{\otimes -1}=\Hom_{\CO}(T_{\pi}\CF,U_\infty),$$
where $T_{\pi}\CF=\varprojlim \CF[\pi^{n+1}]$ is the $\pi$-adic Tate module of $\CF$.
Wintenberger showed that $U_\infty^*$ is a rank $2$ free $\CO[[\Gal(\Phi_\infty/\Phi_0)]]$-module (cf. \cite{Wi}).

Consider the Coates-Wiles logarithmic derivatives
$$\delta:U_\infty^{*}\ra \CO,\quad x=u\otimes a\otimes v^{\otimes -1}\mapsto a\cdot\frac{f'(0)}{f(0)},$$
and
$$\delta_n:U_\infty^*\ra \Phi_n,\quad x=u\otimes a\otimes v^{\otimes -1}\mapsto \frac{a}{\lambda'(v_n)}\cdot\frac{f'(v_n)}{f(v_n)},$$
where $u=(u_n)_n\in \varprojlim U(\Phi_n)$, $a\in \CO$, $v=(v_n)_n\in T_{\pi}\CF$ is a generator as $\CO$-module, $f\in \CO[[X]]^\times$ is the Coleman power series such that $f(v_n)=u_n$ and $\lambda$ is the formal logarithm of $\CF$ normalized by $\lambda'(0)=1$.

For a finite character $\chi:\Gal(\Phi_\infty/\Phi)\ra \ov{\BQ}_p^\times$ which factor through $\Gal(\Phi_n/\Phi)$, we define
$$\delta_\chi:U_\infty^*\ra \ov{\BQ}_p,\quad x\mapsto \frac{1}{\pi^{n+1}}\sum_{\gamma\in\Gal(\Phi_n/\Phi)}\chi(\gamma)\delta_n(x)^\gamma.$$
It is independent of the choice of $n$. For $\sigma\in \Gal(\Phi_\infty/\Phi)$, we have $\delta_\chi(x^\sigma)=\chi(\sigma)^{-1}\delta_\chi(x)$.

\subsubsection{Anticyclotomic projection}
Let $\Psi_\infty$ be the anticyclotomic $\BZ_p$-extension of $\Phi$ and $G^-=\Gal(\Psi_\infty/\Phi)$ be its Galois group.
Let $G^+=\Gal(\Theta_\infty/\Psi_\infty)$.
Let $\Psi_n$ be the subextension of $\Psi_\infty/\Phi$ of degree $p^n$.
If $\chi$ is an anticyclotomic character, i.e., $\chi$ is a homomorphism $\Gal(\Psi_n/\Phi)\ra \ov{\BQ}_p^\times$ for some $n$, then $\delta_\chi((\sigma-1)U_\infty^*)$ vanishes for all $\sigma\in \Gal(\Phi_\infty/\Psi_\infty)$.
Set $V_\infty^*:=U_\infty^*/\{(\sigma-1)u|\sigma\in\Gal(\Phi_\infty/\Psi_\infty),u\in U_\infty^*\}.$
Then $\delta_\chi$ factors through $V_\infty^*$.

\subsubsection{Decomposition of Local Principal Units}
We say a non-trivial anticyclotomic character $\chi$ has conductor $p^n$ if $\chi$ factors through $\Gal(\Phi_{n-1}/\Phi)$ but not through $\Gal(\Phi_{n-2}/\Phi)$, equivalently, $\chi$ factors through $\Gal(\Psi_{n-1}/\Phi)$ but not through $\Gal(\Psi_{n-2}/\Phi)$. We say that trivial character has conductor $1$.

Let $\Xi^+$ (resp. $\Xi^-$) be the set of anticyclotomic characters whose conductors are even (resp. odd) powers of $p$.
Define
$$V_\infty^{*,\pm}:=\left\{v\in V_\infty^*|\delta_\chi(v)=0\text{ for every }\chi\in \Xi^{\mp}\right\}.$$
Set $\Lambda=\CO[[G^-]]$.
It is known that $V_\infty^*$ is a free $\Lambda$-module of rank $2$.
We will show the following theorem.
\begin{thm}\label{thm:main}
Assume $p\geq 3$. We have
$$V_\infty^*\simeq V_\infty^{*,+}\oplus V_\infty^{*,-}.$$
\end{thm}
\begin{remark}
Rubin conjectured and verified the direct decomposition for $5\leq p\leq 1000$ and $p \not\equiv 1\pmod{12}$(cf. \cite{Ru}).
Burungale, Kobayashi and Ota proved it for primes $p\geq 5$(cf. \cite{BKO}).
We modify Burungale-Kobayashi-Ota's proof to include the case $p=3$.
\end{remark}

\subsection{Strategy}
We know that $V_\infty^*\simeq \Lambda^2$ by \cite{Wi}.
Consider the anticyclotomic projections of elliptic units in $V_\infty^{*}$. Their images under $\delta_\chi$ are algebraic parts of $L$-values of Hecke character $\chi\varphi$ (Theorem \ref{thm:Lvalue},(1)).
They vanish if the root number of $\chi\varphi$ is $-1$, which is the case if the root number $W(\varphi)$ of $\varphi$ is $1$ and the conductor of $\chi$ is an odd power of $p$, or $W(\varphi)=-1$ and the conductor of $\chi$ is an even power of $p$.
Hence  the root number of $\varphi$ determines which of $V_\infty^{*,\pm}$ the elliptic units belong to (Theorem \ref{thm:Lvalue},(2)).
Moreover, Rohrlich showed that there are all but finitely many anticyclotomic characters $\chi$ such that $L(\varphi\chi,1)\neq 0$.
This ensures the elliptic units above are nontrivial in $V_\infty^\pm$, so
$$\rank_\Lambda V_\infty^{*,\pm}\geq 1$$ (Theorem \ref{thm:rankgeq1}).

On the other hand, we have a perfect pairing
$$\langle\ ,\ \rangle:\CF(\Psi_\infty)\otimes_\CO\Phi/\CO\times V_\infty^*\ra \Phi/\CO.$$
The annihilator of $V_\infty^{*,\pm}$ under this paring  is $A^{\pm}\otimes\Phi/\CO$, where
$$A^{\pm}:=\{y\in \CF(\Psi_\infty)|\lambda_\chi(y)=0\text{ for all }\chi\in \Xi^{\pm}\}.$$
These modules $A^\pm$ can be well studied (Proposition \ref{prop:Kum} and Lemma \ref{lem:lambda}).
We may found that $A^{\pm}\otimes\Phi/\CO$ generate the whole $\CF(\Psi_\infty)\otimes\Phi/\CO$. Hence
$$V_\infty^{*,+}\cap V_\infty^{*,-}=0 \text{ and }\rank_\Lambda V_\infty^{*,\pm}=1.$$

Now its suffices to show that $V_\infty^*/V_\infty^{*,-}$ is isomorphic to $V_\infty^{*,+}$.
This can be done if $V_\infty^*/V_\infty^{*,-}$ is free of rank $1$ and there is $\xi\in V_\infty^{*,+}\otimes \CR$ for some coefficient ring $\CR$ such that $\delta_{\chi}(\xi)\in \CO^\times\otimes\CR$.

Burungale-Kobayashi-Ota considered the elliptic units of root number $+1$ twisted by an anticyclotomic character  $\nu$ along $\BZ_\ell$-extension for an auxiliary $\ell$.
Their images under $\delta_\chi$ are algebraic parts of $L(1,\varphi\chi\nu)$.
By the work of Finis \cite{Fi}, this $\nu$ can be well-chosen for the purpose that the algebraic parts of $L$-values do not vanish mod $p$.
Hence there is $\xi_\nu\in V_\infty^{*,+}\otimes\CR$ for some coefficient ring $\CR$ such that
$$\delta_\chi(\xi_\nu)\in\CO^\times\otimes\CR$$
(Theorem \ref{thm:V+}).

The last step (see Theorem \ref{thm:V-}) is to show
\begin{equation}\label{-free}
(A^-\otimes\Phi/\CO)^{G^-}=\CF(\Phi)\otimes\Phi/\CO.
\end{equation}
By the Nakayama Lemma, the $\Lambda$-module $V_\infty^*/V_\infty^{*,-}$ is free of rank $1$, which completes the proof.
To show Theorem \ref{thm:V-}, it suffices to prove
$$|\wh{H}^0(G_n^-,A_n^-)|=|\CF(\Phi)/\RN_{n/0}(A_n^-)|\leq p^{n-1},$$
where $A_n^-=A^-\cap \CF(\Psi_n)$.
The key point is to construct points in $A_n^-$ whose norm in $A_0^-$ is locally $p$-indivisible (Theorem \ref{thm:localpts}).
Actually, we will construct points satisfying
\begin{enumerate}
  \item $y\in \CF(\Phi)\backslash p\CF(\Phi)$,
  \item $y_s\in \CF(\Psi_s)$ such that $\tr_{s+1/s}y_{s+1}=-y_{s-1}$ for $s\geq 1$ and $\tr_{1/0}y_1=-y$.
  \item $y_s\in A^{-}$ if $s$ is odd.
\end{enumerate}

Choose a supersingular CM elliptic curve $E$ which has good supersingular reduction at $p$.
Then $\wh{E}\simeq \CF$ over $\CO$, where $\wh{E}$ is the formal group associated to $E$.
Rubin considered the Heegner points in $A_n^-$, which are the images of some CM points on $X_0(N)(\CO)$ under the modular parametrization map
$$\pi: X_0(N)\ra E_{/\CO}.$$
 If the bottom layer is $p$-indivisible, then we are done.
Unfortunately, we do not know the $p$-divisibility of it.

The idea of Burungale-Kobayashi-Ota is to construct formal CM points instead.
There are supersingular points on $X_0(N)(\CO)$ which may not be CM but fake CM, i.e., the formal group of the "representative" elliptic curve has an $\CO$-action. We call such points formal CM points.
Similar to the construction of Heegner points, Gross constructed a system of compatible formal CM points on $\wh{E}(\Psi_n)$.

Now we need to find a "good" supersingular point on $X_0(N)(\CO)$ which leads to the $p$-indivisibility in the bottom layer.
We may choose a point on $X_0(N)(\BF_{p^2})$ such that
\begin{enumerate}
  \item the point "represents" a supersingular elliptic curve with a level structure
  \item under modular parametrization $\ov{\pi}: X_0(N)\ra E$ the point is unramified and maps to $\ov{O}$
\end{enumerate}
Taking formal completion of $\pi$ over $\CO$ along these two points, we get an isomorphism $\wh{X_0(N)}\simeq \wh{E}$.
Choose $Q\in \wh{E}(\fm)\backslash p\wh{E}(\fm)$, where $\fm$ is the maximal ideal of $\CO$. Let $P\in X_0(N)(\CO)$ be the preimage of $Q$.
Then the point $P$ "represents" a fake CM elliptic curve $A$ with a level structure. This elliptic curve $A$ is what we want.
Noting that $X_0(N)$ is not fine moduli, we need replace $X_0(N)$ by $X(\Gamma_0(N)\cap \Gamma_1(M))$ and modify the above argument.

\subsection*{Acknowledgement}
We would like to thank Professor Ashay A. Burungale and Professor Ye Tian for their insightful discussions.
We also appreciate the valuable suggestions provided by the anonymous referee.

\section{Hecke \texorpdfstring{$L$}{L}-values and elliptic units}\label{sec:Lval}

In this section, we recall the proof of the following two theorems given in \cite{Ru} and \cite{BKO}.

\begin{thm}\label{thm:rankgeq1}
$\rank_\Lambda V_\infty^{*,\pm}\geq 1$.
\end{thm}

\begin{thm}\label{thm:V+}
There exists an element $\xi\in V_\infty^{*,+}$ such that $\delta(\xi)\in \CO^\times$.
\end{thm}


The basic ideas involve using the relation of elliptic units and Hecke $L$-values, and properties of Hecke $L$-values proved by Rohrlich \cite{Ro1} and Finis \cite{Fi}.

Firstly, we choose an auxiliary imaginary quadratic field. By \cite[Lemma 3.4]{BKO}, there exist infinitely many imaginary quadratic fields $K$ of odd discriminants such that
\begin{enumerate}
  \item $\left(\frac{2}{D_K}\right)=+1$ where $-D_K<0$ is the discriminant of $K$;
  \item $p$ inerts in $K$ and is prime to $h_K$.
\end{enumerate}
In the rest of our paper, $K$ is an imaginary quadratic field satisfying (2). We do not assume that $K$ satisfies (1) except in the proof of Theorem \ref{thm:V+}.  For
a non-zero integral ideal $\fg$ of $K$, we denote by $K(\fg)$ the ray class field of $K$ of
conductor $\fg$. Let $H=K(1)$ be the Hilbert class field of $K$.

Let $\varphi$ be a Hecke character over $K$ with
infinity type $(1, 0)$ of $\ff_\varphi$ such that $\varphi\circ \RN_{H/K}$ corresponds to an elliptic curve $E_{/H}$ which is CM by $\CO_K$, is isogenous to all its $\Gal(H/\BQ)$-conjugate and is good at primes above $p$. We note that if $\varphi$ is a canonical Hecke character (in the sense of \cite{Ro}), such an $E$ always exists.

We fix a smooth Weierstrass model of the elliptic curve $E$ over $\CO\cap H$ and we may assume the period lattice $L$ attached to the N\'{e}ron differential $\omega$ is given by $\Omega\CO_K$ for some $\Omega\in \BC^\times$.
Fix such $\Omega$.

Let $\ell\geq 5$ be a prime such that $\ell$ splits in $K$, $\ell\nmid h_K$ and $p\nmid \ell-1$.
Let $\fX_\ell$ be the set of finite Hecke characters that factor through the anticyclotomic $\BZ_\ell$-extension of $K$.

\begin{thm}\label{thm:Lvalue}
Let $\nu\in \fX_\ell$ be a character of order $\ell^m$.
Let $\CR$ be the integer ring of the finite extension of $\Phi$ generated by the image of $\varphi$ and $\nu$.
Then there exists a $\xi_\nu\in U_\infty^*\otimes\CR$ such that
\begin{enumerate}
  \item the following holds $$\delta(\xi_\nu)=\frac{L_{\ff\ell}(\ov{\varphi}\nu,1)}{\Omega},\quad \delta_\chi(\xi_\nu)=\frac{L_{\ff\ell p}(\ov{\varphi}\nu\chi,1)}{\Omega},$$
      for all finite characters $\chi$ of $\Gal(\Phi_\infty/\Phi_0)$;
  \item the anticyclotomic projection of $\xi_\nu$ lies in $V_\infty^{*,\epsilon}\otimes\CR$ where $\epsilon$ is the root number of $\varphi$.
\end{enumerate}
\end{thm}
\begin{proof}
As above, we choose an imaginary quadratic field $K$, a prime $p$ that inerts in $K$ and is prime to $h_K$, a $\BQ$-curve $E$ and the associated Hecke character $\varphi$.
Besides, we choose an auxiliary prime $\ell$.
Let $T=T_{\pi}E$.
\begin{enumerate}
  \item Consider the elliptic units $z_{\ff l^m}=(z_{\ff l^m p^{n}})_n\in \varprojlim_{n}H^1(K(\ff l^m p^{n}),T^{\otimes -1}(1))$.
  Let $$M_n=H(E[p^{n+1}])L_m\subset K(\ff \ell^m p^{n+1})$$ where $L_m$ is the $m$-th layer of anticyclotomic $\BZ_\ell$-extension of $K$.
  Let $\nu$ be an anticyclotomic Hecke character over $K$ of order $\ell^m$.
  Consider the composition of the following maps
  $$\begin{aligned}
  &\varprojlim_{n}H^1(K(\ff\ell^m p^{n+1}),T^{\otimes -1}(1))\xrightarrow{\mathrm{cores}}\varprojlim_{n}H^1(M_n,T^{\otimes -1}(1))\\
  &\xrightarrow{\loc_p}\varprojlim_n H^1(M_n\otimes K_p,T^{\otimes -1}(1))\xrightarrow{\nu}\varprojlim_{n}H^1(H(E[p^{n+1}])\otimes K_p,T^{\otimes -1}(1))\otimes\CR\\
  &\xrightarrow{\pr} \varprojlim_{n}H^1(\Phi_n,T^{\otimes -1}(1))\otimes\CR\ra\left(\varprojlim_n H^1(\Phi_n,T^{\otimes -1}(1))\right)^{\Delta}\otimes\CR
  \simeq U_\infty^*\otimes\CR.
  \end{aligned}$$
  Let $\xi_\nu\in U_\infty^*\otimes\CR$ be the image of $z_{\ff\ell^m p^{n+1}}$ under the above map.
  Then we have
  $$\delta_\chi(\xi_\nu)=\frac{L_{\ff\ell p}(\ov{\varphi}\chi\nu,1)}{\Omega}$$
  for all finite characters $\chi$ of $\Gal(\Phi_\infty/\Phi_0)$.

  \item For character $\chi$ of $G^-$ of conductor $p^{n+1}$, Greenberg (\cite[p.247]{Gr}) showed that $W(\ov{\varphi}\nu\chi)=W(\ov{\varphi}\chi)=(-1)^{n+1}W(\ov{\varphi})$ if $p$ is odd and $\ell\nmid \ff$, .
      Therefore $L(\ov{\varphi}\nu\chi,1)=0$ if $(-1)^{n+1}W(\varphi)=-1$ and the theorem follows from (1).
\end{enumerate}
\end{proof}

\begin{proof}[Proof of Theorem \ref{thm:rankgeq1}]
By \cite{Ro1}, for all but finitely many anticyclotomic characters $\rho$,
$$L(1,\rho\varphi)\neq 0,\text{ if }W(\varphi\rho)=1.$$
If $\rho$ is of conductor $p^n$ and $\varphi$ is of root number $\epsilon$, then $W(\varphi\rho)=(-1)^n\epsilon$. Thus there exist infinitely many anticyclotomic characters $\rho$ such that $L(1,\varphi\rho)\neq 0$.
Hence $\delta_\chi(\xi)\neq 0$ for the elliptic units $\xi$ associated to $\varphi$ by the theorem above.
Since $V_\infty^{*}\simeq \Lambda^2$ is torsion-free, we have $\rank_\Lambda V_\infty^{*,\pm}\geq 1$.
\end{proof}

\begin{proof}[Proof of Theorem \ref{thm:V+}]
Suppose that $\varphi$ is canonical. We have $W(\varphi)=+1$ (cf. \cite{Ro}).
Then by \cite{Fi}, for all but finitely many $\nu\in \fX_\ell$, one has
$$\Omega^{-1}L_{\ff}(\ov{\varphi}\nu,1)\in \CR^\times.$$
Fix a $\nu$, Theorem \ref{thm:Lvalue} shows that there is a $\xi_\nu\in V_\infty^{*,+}\otimes\CR$ such that $\delta(\xi_\nu)\in \CR^\times$.
It implies that there exists an element of $V_\infty^{*,+}$ whose image under $\delta$ belongs to $\CO^\times$ .
\end{proof}

\section{Kummer pairing}
We recall the construction of the Kummer pairing $$\langle\ ,\ \rangle: \left(\CF(\Psi_\infty)\otimes_\CO\Phi/\CO\right) \times V_\infty^*\ra \Phi/\CO.$$
Note that $\Theta_n=\Phi_n^{\Delta}$ for all $n\leq \infty$.
  The Kummer sequence
  $$0\ra \CF[\pi^{n+1}]\ra \CF(\ov{\Phi})\xrightarrow{\pi^{n+1}} \CF(\ov{\Phi})\ra 0 $$
  gives us an exact sequence
  $$0\ra \CF(\Theta_n)/\pi^{n+1}\CF(\Theta_n)\ra H^1(\Theta_n,\CF[\pi^{n+1}])\ra H^1(\Theta_n,\CF(\ov{\Phi}))[\pi^{n+1}]\ra 0.$$
  Hazewinkel \cite{Haz} showed that $\cap_n \RN_n\CF(\Theta_n)=0$ if $\CF$ is a Lubin-Tate formal group of height $2$ over $\CO$.
  Hence $\varprojlim \CF(\Theta_n)=0$ and its Tate duality (\cite{Tat}) $\varinjlim H^1(\Theta_n,\CF(\ov{\Phi}))_{p^{n+1}}$ is also zero.
  Taking direct limit of the above exact sequences, we have
  $$\CF(\Theta_n)\otimes\Phi/\CO\simeq H^1(\Theta_\infty,\CF[\pi^\infty])\simeq \Hom(\Gal(\ov{\Phi}/\Phi_\infty),\CF[\pi^\infty])^\Delta\simeq \Hom_\CO(U_\infty,\CF[\pi^\infty]),$$
where the last isomorphism is given by local class field theory.
Therefore we have a perfect pairing
$$\langle\ ,\ \rangle: \left(\CF(\Theta_\infty)\otimes\Phi/\CO\right) \times U_\infty^*\ra \Phi/\CO. $$
Since $\CF(\Theta_\infty)$ has no $p$-torsion, the exact sequence
$$0\ra \CF(\Theta_\infty)\ra \CF(\Theta_\infty)\otimes_\CO\Phi\ra \CF(\Theta_\infty)\otimes \Phi/\CO\ra 0$$
induces an isomorphism $\left(\CF(\Theta_\infty)\otimes\Phi/\CO\right)^{G^+}\simeq \Hom_\CO(V_\infty,\CF[\pi^\infty])$.
However, we have that
$$\left(\CF(\Theta_\infty)\otimes\Phi/\CO\right)^{G^+}/(\CF(\Psi_\infty)\otimes\Phi/\CO)\simeq H^1(G^{+},\CF(\Theta_\infty))\subset H^1(\Psi_\infty,\CF(\ov{\Phi}))=\varinjlim H^1(\Psi_n,\CF(\ov{\Phi}))=0.$$
Here the reason for the last equality is similar to $\varinjlim H^1(\Psi_n,\CF(\ov{\Phi}))=0$.
Hence we have a perfect paring
$$\langle\ ,\ \rangle:\left(\CF(\Psi_\infty)\otimes_\CO\Phi/\CO\right)\times V_\infty^*\ra \Phi/\CO.$$
By Wiles' explicit reciprocity law (\cite{Wil}), the pairing can be described as
$$\langle y\otimes\pi^{-n},x\rangle=\pi^{-1-m-n}\Tr_{\Phi_m/\Phi}(\delta_m(x)\lambda(y))\in\Phi/\CO$$
with  $y\in \CF(\Psi_n)$, $x\in V_\infty^*$ and some sufficiently large $m$.

For any anticyclotomic character $\chi$ of conductor dividing $p^{n+1}$, let
$$\lambda_\chi:\CF(\Psi_\infty)\ra \Phi_\infty,\quad y\mapsto \frac{1}{\pi^n}\sum_{\gamma\in \Gal(\Psi_n/\Phi)}\chi^{-1}(\gamma)\lambda(y)^\gamma,\quad y\in \CF(\Psi_n).$$
Denote
$$A^{\pm}:=\{y\in \CF(\Psi_\infty)|\lambda_\chi(y)=0\text{ for all }\chi\in \Xi^{\pm}\}.$$

We recall the following properties of $\lambda_\chi$.
\begin{lem}[{\cite[Lemma 5.5]{Ru}}]\label{lem:lambda}\hfill
\begin{enumerate}
  \item If $y\in \CF(\Psi_\infty)$, $\chi$ is a finite character of $G^-$ and $\sigma\in G^-$, then $\lambda_\chi(y^\sigma)=\chi(\sigma)\lambda_\chi(y)$;
  \item If $y\in \CF(\Psi_n)$ and the conductor of $\chi$ is greater than $p^{n+1}$, then $\lambda_\chi(y)=0$;
  \item If $y\in \CF(\Psi_\infty)$, then $\lambda(y)=\sum\lambda_\chi(y)$, summing over all finite characters $\chi$ of $G^-$;
  \item If $m\geq n$, $y\in \CF(\Psi_m)$ and $\chi$ is a character of $\Gal(\Psi_n/\Phi)$, then $\lambda_\chi(\RN_{m/n}y)=p^{m-n}\lambda_\chi(y)$.
  \item $A^+\cap A^-=0$;
  \item $(A^+\otimes \Phi/\CO) + (A^-\otimes\Phi/\CO)=\CF(\Psi_\infty)\otimes\Phi/\CO$.
\end{enumerate}
\end{lem}
%
%
%
%
%

\begin{prop}[{\cite[Proposition 5.6]{Ru}}]\label{prop:Kum}
Under the Kummer pairing $(\CF(\Psi_\infty)\otimes_\CO\Phi/\CO)\times V_\infty^*\ra \Phi/\CO$, the annihilator of $V_\infty^{*,\pm}$ is $A^\pm\otimes\Phi/\CO$.
\end{prop}
\begin{proof}
 If $y\in \CF(\Psi_\infty)$ and $x\in V_\infty^*$, then the above formula and Lemma \ref{lem:lambda} yields
      $$\begin{aligned}
      \langle y\otimes \pi^{-n},x\rangle &=\pi^{-1-m-n}\Tr_{\Phi_m/\Phi}(\delta_m(x)\lambda(y))=\pi^{-1-m-n}\sum_{\gamma}\delta_m(x)^\gamma\lambda(y^\gamma)\\
      &=\pi^{-1-m-n}\sum_{\gamma}\delta_m(x)^\gamma\sum_\chi\lambda_\chi(y^\gamma)=\sum_\chi\pi^{-1-m-n}\sum_\gamma\delta_m(x)^\gamma\chi(\gamma)\lambda_\chi(y)\\
      &=\sum_\chi\delta_\chi(x)\lambda_\chi(y).
      \end{aligned}$$
      By definition, $V_\infty^{*,\pm}$ annihilate $A^{\pm}\otimes\Phi/\CO$.

      Now suppose that $x\in V_\infty^*$ and $x$ annihilates $A^{\pm}\otimes\Phi/\CO$ and $\chi\in \Xi^\mp$.
      Choose $y\in A^\pm$ such that $\lambda_\chi(y)\neq 0$.
      Then the above computation shows that
      $$\sum_\rho \delta_\rho(x)\lambda_\rho(y^\gamma)=0$$
      for every $\gamma$.
      Thus
      $$\pi^n\delta_\chi(x)\lambda_\chi(y)=\sum_\rho\sum_\gamma\chi^{-1}(\gamma)\delta_\rho(x)\lambda_\rho(y^\gamma)=0.$$
      Hence $\delta_\chi(x)=0$, i.e., $x\in V_\infty^{*,\pm}$.
\end{proof}

Now we have the following corollary by Lemma \ref{lem:lambda} (6) and Proposition \ref{prop:Kum}.

\begin{cor}\label{cor1}
$V_\infty^{*,+}\cap V_\infty^{*,-}=0$.
\end{cor}

\section{Local points}

In this section, we will prove the following theorem.
\begin{thm}\label{thm:V-}
We have
$$(A^-\otimes\Phi/\CO)^{G^-}=\CF(\Phi)\otimes\Phi/\CO.$$
\end{thm}
\begin{cor}\label{cor3}
The $\Lambda$-module $V_\infty^*/V_\infty^{*,-}$ is free of rank one.
\end{cor}
\begin{proof}
Note that
$$\CF(\Phi)\otimes\Phi/\CO\simeq (A^-\otimes\Phi/\CO)^{G^-}\simeq \Hom\left((V_\infty^*/V_\infty^{*,-})/(\gamma-1),\Phi/\CO\right)$$
where $\gamma$ is the topological generator of $G^-$.
Hence $(V_\infty^*/V_\infty^{*,-})/(\gamma-1) \simeq \CO$ generated by one element.
By Nakayama's lemma, $V_\infty^*/V_\infty^{*,-}$ is also generated by one element.
Hence the $\Lambda$-module $V_\infty^*/V_\infty^{*,-}$ is free of rank one.
\end{proof}

We will construct a system of local points in $\CF(\Psi_n)$, which can be used to show that $(A^-\otimes\Phi/\CO)^{G^-}$ is isomorphic to the divisible module $\CF(\Phi)\otimes\Phi/\CO$. So the dual module (under Kummer pairing) $V_\infty^*/V_\infty^{*,-}$ is free.

Let $E$ be an elliptic curve over $\BQ$ with good supersingular reduction at $p$.
Consider the modular parametrization $\pi: X_0(N)\ra E$ over $\BQ$.
We may assume $\pi$ is strong Weil by choosing $E$ in its isogeny class.
By the N\'eron mapping property, $\pi$ extends to a morphism between smooth models over $\BZ_p$.

\subsection{A special supersingular elliptic curve}
\begin{lem}[{\cite[Lemma 5.1]{BKO}}]\label{lem:LT}
Let $q=p^2$ and $\ov{A}$ be an elliptic curve over $\BF_q$ with $a_q(\ov{A})=\pm 2p$.
\begin{enumerate}
  \item Any finite subgroup $\ov{A}(\ov{\BF}_q)$ is defined over $\BF_q$.
  \item For $A$ an elliptic curve over $\CO$ which is a lift of $\ov{A}$, the associated formal group $\wh{A}$ is Lubin-Tate with parameter $\mp p$.
\end{enumerate}
\end{lem}
\begin{lem}\label{lem:super-pt}
If $p\geq 3$, there is a supersingular point with $a_{p^2}=\pm 2p$ in $X_0(N)(\BF_{p^2})$ which is unramified under $\ov{\pi}:X_0(N)_{\BF_{p^2}}\ra \ov{E}$.
\end{lem}

\begin{proof}
See \cite{BKO} for $p>3$. We give a proof for $p=3$.
Let $S_{ram}$ be the set of points of $X_0(N)$ which are ramified under $\ov{\pi}$.
By Hurwitz formula \cite[Chapter 7, Theorem 4.16]{Liu}
$$\#S_{ram}\leq 2g-2,$$
where $g$ is the genus of $X_0(N)$.
Let $\mu=N\prod_{p\mid N}(1+p^{-1})$ be the degree of natural projection $X_0(N)\ra X(1)$.
By genus formula
$$g=1+\frac{\mu}{12}-\frac{\varepsilon_2}{4}-\frac{\varepsilon_3}{3}-\frac{\varepsilon_\infty}{2},$$
where $\varepsilon_2$ (resp. $\varepsilon_3$) is the number of elliptic points of period $2$ (resp. $3$) in $X_0(N)$, and $\varepsilon_\infty$ the number of cusp of $X_0(N)$.
Hence
$$\#S_{ram}\leq\frac{\mu}{6}-\frac{\varepsilon_2}{2}-\frac{2\varepsilon_3}{3}-\varepsilon_\infty<\frac{\mu}{6}.$$

The elliptic curve
$$\ov{A}_{/\BF_3}:y^2=x^3-x$$
is supersingular and $j(\ov{A})=0=1728$.
Note that $\ov{A}(\BF_3)=\{O,(0,0),(1,0),(-1,0)\}\simeq \BZ/2\BZ\times\BZ/2\BZ$, $\ov{A}(\BF_9)\simeq \BZ/4\BZ\times\BZ/4\BZ$, and $a_3(\ov{A})=0, a_9(\ov{A})=-6$.
Since $p\nmid N$, the group $\ov{A}[N]$ is isomorphic to $\BZ/N\BZ\times\BZ/N\BZ$ and thus has $\mu$ cyclic subgroup of order $N$, which we denote by $\{C_1(N),\cdots,C_\mu(N)\}$ (they are defined over $\BF_9$).
Since $\#\Aut(\ov{A})=12$ (\cite[Theorem III.10.1]{Si}) and $-1$ induces an isomorphism of pairs $(\ov{A},C_i(N))\ra (\ov{A},C_i(N))$, there are at least $\frac{\mu}{6}$ isomorphism classes of pair $(\ov{A},C_i(N))$.
Hence there is a supersingular point with $a_{9}=-6$ in $X_0(N)(\BF_9)$ which is unramified under $\ov{\pi}:X_0(N)\ra \ov{E}$.

\end{proof}

\subsection{A formal CM point}
By Lemma \ref{lem:super-pt}, we can choose  a supersingular point $\ov{P}$ of $X_0(N)_{\BF_q}$ unramified under $\ov{\pi}$, representing an elliptic curve $\ov{A}$ with $a_{p^2} = \pm 2p$ and a $\Gamma_0(N)$-level structure. In particular, when $p=3$, $\ov{A}$ is chosen to be $y^2=x^3-x$.
We assume that $\ov{\pi}(\ov{P})=\ov{O}$  by replacing $\ov{\pi}$ with $\# E(\BF_{p^2})\ov{\pi}$.
In this subsection, for a properly chosen $E$, we construct a lift $P\in X_0(N)(\CO)$ of $\ov{P}$ representing an elliptic curve $A$ over $\CO$ and a $\Gamma_0(N)$-level structure, such that $\pi(P)\in \wh{E}(\fm)\backslash p\wh{E}(\fm)$, where $\fm$ be the maximal ideal of $\CO$.
For $p>3$, the construction details can be found in \cite{BKO}.
From now on, we assume $p=3$.
\begin{lem}\label{lem:key}
Let $\ov{A}:y^2=x^3-x$ be the supersingular elliptic curve over $\BF_3$.
There are infinitely many integers $N$ such that
\begin{enumerate}
  \item $N$ is the conductor of a CM elliptic curve $E_{/\BQ}$ which is good at $2$, $3$ and satisfies $a_3(E)=0$;
  \item for a $\Gamma_0(N)$-structure $C(N)$ of $\ov{A}$, the automorphism group of $(\ov{A},C(N))$ over $\ov{\BF}_3$ is $\{\pm 1\}$.
\end{enumerate}
\end{lem}
\begin{proof}
Let $X$ be the set of integers satisfying conditions a) and b) in the lemma, and $Y$  the set of integers $N$ satisfying the following conditions:
\begin{enumerate}
  \item[a)] $N$ is conductor of a CM elliptic curve $E_{/\BQ}$ which is good at $2,3$ and satisfies $a_3=0$;
  \item[b)] $\varphi(N)>\#(\ker(g+1)\cup\ker(g-1))$;
  \item[c)] $-3$ is not a square in $\BZ/N\BZ$.
\end{enumerate}
We claim that $Y$ is an infinite set and $Y\subset X$. It completes the proof.

We first show that $Y$ is an infinite set.
Choose a CM elliptic curve $E_{/\BQ}$ which has good reduction at $2,3$ and satisfies $a_3=0$ (for example, $E_{/\BQ}:y^2+y=x^3-38x+90$).
Let $N_E$ be its conductor.
By Dirichlet's theorem on arithmetic progressions, there are infinitely many primes $\ell\equiv 5\pmod{12}$ prime to $N_E$ and satisfying $\varphi(\ell^2N)>\#(\ker(g+1)\cup\ker(g-1))$.
Let $E^\ell$ be the quadratic twists of $E$ by prime $\ell$.
It is a CM elliptic curve with conductor $\ell^2N_E$ and satisfies $a_3=0$.
Hence $\ell^2N_E\in Y$, which implies that $Y$ is a infinite set.

Now we show that $Y\subset X$.  Let $N$ be an integer satisfying b) and c). Since $\# \Aut(\ov{A})=12$,  it suffices to prove that for any $g\in\Aut(\ov{A})\backslash\{\pm 1\}$ of order $2$ or $3$, the actions of $g$ on $N$-cyclic subgroup of $\ov{A}$ are not stable.
If not, i.e., there is $g\in \Aut(\ov{A})\backslash\{\pm 1\}$ of order $2$ or $3$ and an $N$-cyclic subgroup $C(N)$, such that for any primitive elements $\alpha\in C(N)$,  $g\alpha=n\alpha$ for some $n\in\BZ/N\BZ$.
It follows that $n^2\alpha=g^2\alpha=\alpha$ or $n^3\alpha=g^3\alpha=\alpha$  (depends on the order of $g$).
Thus $0\equiv n^2-1\equiv (n-1)(n+1)\pmod{N}$ or $0\equiv n^3-1\equiv (n-1)(n^2+n+1)\pmod{N}$.
Since $-3$ is not a square in $\BZ/N\BZ$, we have $n\equiv1$ or $-1$.
So all primitive elements $\alpha\in C(N)$ must belong to $\ker(g-1)\cup\ker(g+1)$, which contradicts condition b).

\end{proof}

Choose a point $\xi$ of order $4$ in $\ov{A}(\BF_{p^2})\simeq \BZ/4\BZ\times\BZ/4\BZ$. Let $X(\Gamma_0(N),\Gamma_1(4))$ be the modular curve  with $\Gamma_0(N)$ and $\Gamma_1(4)$-level structure.
Then $X(\Gamma_0(N),\Gamma_1(4))$ is a fine moduli space. Consider
$$\xymatrix{
X(\Gamma_0(N),\Gamma_1(4))(\BF_{p^2})\ar[d]_{\ov{\pi}'} & \ov{P}'=(\ov{A},C(N),\xi)\ar@{|->}[d] \\
X_0(N)(\BF_{p^2})\ar[d]_{\ov{\pi}} & \ov{P}=(\ov{A},C(N)) \ar@{|->}[d]^{\text{unramified}}\\
E(\BF_{p^2}) & \ov{O}
}.$$
We choose $E$ as in Lemma \ref{lem:key}.
Then the automorphism group of $(\ov{A},C(N))$ is $\{\pm 1\}$.
Hence $\#{\ov{\pi}'}^{-1}(\ov{P})=\deg \ov{\pi}'=[\GL_2(\BZ):\Gamma_1(4)]/2$, and therefore $\ov{\pi}'$ is unramified at $\ov{P}'$.
The formal completion of $\pi\circ\pi':X(\Gamma_0(N),\Gamma_1(4))\ra E$ (on integral models) at $\ov{P}'$ is an isomorphism (\cite[Chapter 4, Proposition 3.26]{Liu}).

Take a point
$$Q\in \wh{E}(\fm)\backslash p\wh{E}(\fm).$$
Then there is a point $P'\in X(\Gamma_0(N),\Gamma_1(4))(\CO)$ over $\ov{P}'$ sent to $Q$ by $\pi\circ\pi'$.
As $X(\Gamma_0(N),\Gamma_1(4))$ is a fine moduli space, there is an elliptic curve $A$ defined over $\CO$ that represents $P'$ by the moduli interpretation. The formal group $\hat{A}$ is Lubin–Tate
by Lemma \ref{lem:LT}. In particular, A is a formal CM elliptic curve.
Let $P$ be the image of $P'$ in $X_0(N)$.

%

\subsection{Construction of local points} Since $\wh{A}$ is Lubin-Tate, the module $T=T_pA=\CO t$ is a free $\CO$-module of rank $1$.
For $s\geq 0$, let $T_s=p^{-s}\BZ_p t+T$, $C_s=T_s/T$.
Let $A_s=A/C_s$, a quasi-canonical lift of conductor $p^s$ of $\ov{A}$ with respect to $A$.

Let $\Psi_s'$ be the fixed field of subgroup of $\Gal(\ov{\Phi}/\Phi)$ stabilizing $T_s$ and $\Psi_\infty '=\cup \Psi_s '$.
It's known that
$$\Gal(\Psi_s'/\Phi)=(\CO/p^s\CO)^\times/(\BZ/p^s\BZ)^\times,\quad \Gal(\Psi_\infty'/\Phi)\simeq \BZ_p\times \BZ/(p+1)\BZ.$$
Let $\Delta'$ be the torsion subgroup of $\Gal(\Psi_\infty'/\Phi)$.
The field $\Psi_\infty'$ contains the anticyclotomic $\BZ_p$-extension $\Psi_\infty$.
The field $\Psi_s$ lies in $\Psi_{s+1}'$.

Then $A_s/\CO_{\Psi_s'}$ and the canonical level structure induced from that of $A$ define a point $z_s\in X_0(N)(\CO_{\Psi_s'})$.
Let $x_s=\pi(z_s)$.
Let
$$y_s=\sum_{\sigma\in \Delta'}\sigma x_{s+1}\in \wh{E}(\fm_{\Psi_{s}}), \text{ and }y=(p+1)Q\in \wh{E}(\fm).
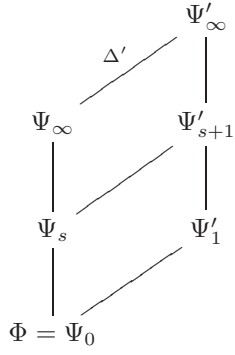
$$

$$\xymatrix{
 & \Psi_\infty'\ar@{-}[ld]_{\Delta'}\ar@{-}[d]\\
\Psi_\infty \ar@{-}[d] & \Psi_{s+1}'\ar@{-}[d]\ar@{-}[ld]\\
\Psi_s \ar@{-}[d] & \Psi_1'\ar@{-}[ld]\\
\Phi=\Psi_0 &
}$$

\begin{thm}\label{thm:localpts}
There is a system of local points $y_s\in \CF(\Psi_s)$ and $y\in \CF(\Phi)\backslash p\CF(\Phi)$, such that
$$\Tr_{s+1/s}y_{s+1}=a_p y_s-y_{s-1},\quad s\geq 1$$
and
$$\Tr_{1/0}y_1=a_p y_0-y,\quad \text{ and }y_0=a_px_0,$$
where $a_p=a_p(E)(=0)$.
Moreover, $y_s\in A^+$ if $s$ is even, and $y_s\in A^-$ if $s$ is odd.
\end{thm}

\begin{proof}
Identify $\CF$ with $\wh{E}$.
Consider the action of the Hecke operator $T_p$ on $x_s$.
There are two types of lattice containing $T_s$ with index $p$:
$$\frac{1+ap^s}{p^{s+1}}\BZ_p t+T\text{ for }a\in \{0,1,\dots,p-1\},\text{ or }\frac{1}{p^{s}}\BZ_p t+\frac{1}{p}T.$$
The first type is of form $
\sigma x_{s+1}$ and permuted by the action of $\Gal(\Psi_{s+1}'/\Psi_{s}')$ and the second type is equivalent to the lattice $\frac{1}{p^{s-1}}\BZ_p t+T$.
Hence for $s\geq 1$, we have
$$T_px_s=\sum_{\sigma}\sigma x_{s+1}+x_{s-1}.$$
Since $T_p$ acts as $a_p(E)$ on $E$, we have the desired relation.

For the proof of $y_s\in A^\pm$, consider the anticyclotomic character $\chi$ of conductor $p^{k+1}$ for $k\geq 1$.
If $s<k$, then $\lambda_\chi(y_s)=0$.
If $s\geq k$, then $\lambda_\chi(\RN_{s/k}y_s)=p^{s-k}\lambda_\chi(y_s)$.
But if $2\nmid s-k$, we have $$\lambda_\chi(\RN_{s/k}y_s)=\lambda_\chi\left(-(-p)^{(s-k-1)/2}y_{k-1}\right)=0,$$
i.e. $\lambda_\chi(y_s)=0$, hence $y_s\in A^-$ if $s$ is odd.
Similarly, if $\chi$ is trivial and $s$ is even,
$$p^s\lambda_\chi(y_s)=\lambda_\chi(\RN_{n/0}y_s)=\lambda_\chi\left((-p)^{n/2}y_0\right)=0.$$
Hence $y_s\in A^+$ if $s$ is even.
\end{proof}

\subsection{Proof of Theorem \ref{thm:V-}}
Write $G_n^-=\Gal(\Psi_n/\Phi)$.
For any $\CO[G_n^-]$-module $Z$, denote the Herbrand quotient of $Z$ by $h_n(Z)$, i.e.,
$$h_n(Z):=|\wh{H}^0(G_n^-,Z)|/|H^1(G_n^-,A^-)|.$$
We know that $h_n(Z_1/Z_2)=h_n(Z_1)/h_n(Z_2)$ and $h_n(Z)=1$ if $Z$ is finite.
Let $A_n^-=A^-\cap \CF(\Psi_n)$.
The exact sequence
$$0\ra A^-\ra A^-\otimes\Phi\ra A^-\otimes\Phi/\CO\ra 0$$
gives the $\CO[G^-]$-mod isomorphism
$$H^1\left(G^-,A^-\right)\simeq\left(A^-\otimes\Phi/\CO\right)^{G^-}/\left((A^-)^{G^-}\otimes\Phi/\CO\right)\simeq \left(A^-\otimes\Phi/\CO\right)^{G^-}/\left(\CF(\Phi)\otimes\Phi/\CO\right).$$
Note that for odd $n$, we have $h_n(A_n^-)=p^{n-1}$ (\cite[Lemma 7.1]{Ru}), hence
$$|H^1(G_n^-,A_n^-)|=|\wh{H}^0(G_n^-,A_n^-)|/h_n(A_n^-)=p^{-(n-1)}|(A_n^-)^{G_n^-}/\Tr_n A_n^-|\leq [\CF(\Phi):\CO y]=1.$$
Therefore, $H^1(G^-,A^-)=\varinjlim H^1(G_n^-,A_n^-)=0$, i.e. $(A^-\otimes\Phi/\CO)^{G^-}=\CF(\Phi)\otimes\Phi/\CO$.

\subsection{Rubin's conjecture}
\begin{thm}
Assuming $p\geq 3$, we have
$$V_\infty^*\simeq V_\infty^{*,+}\oplus V_\infty^{*,-}.$$
\end{thm}
\begin{proof}
The Corollary \ref{cor1}, Theorem \ref{thm:V+} and Corollary \ref{cor3} complete the proof.
\end{proof}

\section{Some applications}
Recall that $K$ is an imaginary quadratic field where $p$ does not divide $h_K$ and is inert in $K$.
Let $K_\infty$ be the anticyclotomic $\BZ_p$-extension of $K$.
We identify $G^-$ with $\Gal(K_\infty/K)$.
Let $\CR$ be the ring of integers of a finite extension of $\Phi$ containing the image of $\wh{\varphi}$.
Let $T=\CR(\wh{\varphi})$ and $W=T\otimes_{\CO}\Phi/\CO$.
The completion of $K_n$ at the prime above $p$ is identical to $\Psi_n$.
Note that $W\simeq \CF[\pi^\infty]\otimes\CR$ as a $\CR[G_{\Phi}]$-module.
The exact sequence
$$0\ra \CF[\pi^{n+1}]\ra \CF(\ov{\Phi})\xra{\pi^{n+1}} \CF(\ov{\Phi})\ra 0 $$
gives the Kummer map $\CF(\Psi_n)/\pi^{n+1}\ra H^1(\Psi_n,\CF[\pi^{n+1}])$.
Hence we have
$$\CF(\Psi_n)\otimes\CR\otimes\BQ_p/\BZ_p\ra H^1(\Psi_n,\CF[\pi^\infty])\otimes\CR\simeq H^1(\Psi_n,W).$$
Let $H^1_{\pm}(\Psi_n,W)\subset H^1(\Psi_n,W)$ be the Kummer image of $\CF^{\pm}(\Psi_n)\otimes\CR\otimes\BQ_p/\BZ_p$ where
$$\CF^{\pm}(\Psi_n):=\{y\in\CF(\Psi_n)|\lambda_\chi(y)=0\text{ for all $\chi\in \Xi^{\pm}$ factor through }\Gal(\Psi_n/\Psi)\}.$$
Let $H^1_{\pm}(\Psi_n,T)\subset H^1(\Psi_n,T)$ be the orthogonal complement of $H^1_{\pm}(\Psi_n,W)$ with respect to the local Tate pairing.

We define
$$\Sel_{\pm}(K_n,W)=\ker\left\{H^1(K_n,W)\ra \frac{H^1(\Psi_n,W)}{H^1_{\pm}(\Psi_n,W)}\times\prod_{v\nmid p}H^1(K_{n,v},W)\right\}.$$
Let $\CX_{*}$ be the Pontryagin dual of $\varinjlim_n \Sel_{*}(K_n,W)$ for $*\in \{+,-\}$.
In \cite[Theorem 3.6]{AH} it is shown that $\CX_{\epsilon}$ is a finitely generated torsion $\Lambda$-module.




Let $E$ and $\varphi$ be as defined in section \ref{sec:Lval}.
As in Theorem \ref{thm:Lvalue}, there is a unit $\xi=\xi(E,\Omega)\in U_\infty^*$ such that
$$\delta(\xi)=\frac{L(\varphi,1)}{\Omega}$$
and
$$\delta_\chi(\xi)=\frac{L(\ov{\varphi}\chi,1)}{\Omega}$$
for a finite character $\chi$ of $\Gal(\Phi_\infty/\Phi_0)$.
Let $\epsilon\in \{+,-\}$ be the sign of $\varphi$.
It is known that the projection of $\xi$ on $V_\infty^*$ belongs to $V_\infty^{*,\epsilon}$.
Define $\CC_\infty$ as the free $\Lambda$-submodule of $V_\infty^{*,\epsilon}$ generated by $\xi$.
Take a generator $v_\epsilon$ of the $\Lambda$-module $V_\infty^{*,\epsilon}$ and write
$$\xi=\CL_p(\varphi,\Omega,v_\epsilon)\cdot v_\epsilon$$
for a power series $\CL_p(\varphi,\Omega,v_\epsilon)\in\Lambda$.
We call it Rubin's $p$-adic $L$-function associated with $\varphi$. We sometimes omit the indices of $\CL_p(\varphi,\Omega,v_\epsilon)$ and write its evaluation at an anticyclotomic character $\chi$ by $\CL_p(\chi)$ for simplicity.
Rubin's $p$-adic $L$-function has the following interpolation property:
$$\CL_p(\chi)=\frac{1}{\delta_\chi(v_\epsilon)}\frac{L(\ov{\varphi\chi},1)}{\Omega}$$
In analogy with \cite{BKO}, we have the following theorems.
\begin{thm}
Let $\epsilon =W(\varphi)$ be the sign of $\varphi$, then
$$\chaR(\CX_{-\epsilon})=(\CL_p).$$
\end{thm}



\begin{thm}
Let $\chi$ be an anticyclotomic character of conductor $p^n$.
Then we have
$$\rank E(K_n)^\chi\leq  \begin{cases}
    \ord_{\chi}(\CL_p),  &\chi\in\Xi^\epsilon\\
    \ord_{\chi}(\CL_p)+1,  &\chi\in\Xi^{-\epsilon}
\end{cases}$$
\end{thm}



\section*{Declarations}
\subsection*{Ethical Approval}
Not applicable.
\subsection*{Funding}
Not applicable.

\bibliographystyle{abbrv}
\bibliography{Rubin_conjecture}
\end{document}